\documentclass[12pt,reqno]{amsart}
\usepackage{fullpage,amscd,amssymb,amsthm,eucal}
\usepackage{graphicx,psfrag}
\numberwithin{equation}{section}
\newtheorem{theorem}{Theorem}
\newtheorem{corollary}{Corollary}
\newtheorem{lemma}{Lemma}
\newtheorem{proposition}{Proposition}
{\theoremstyle{definition}\newtheorem{definition}{Definition}}
{\theoremstyle{definition}\newtheorem{remark}{Remark}}
\newcommand{\dd}[2]
 {\left.\frac{\mathrm{d}}{\mathrm{d}#1}\right|_{#1=#2}}
\newcommand{\dee}{\mathrm{d}}
\newcommand{\R}{\mathbb{R}}
\newcommand{\Arccos}{\mathrm{Arccos}}
\newcommand{\tr}{\mathrm{tr}}
\newcommand{\Hol}{\mathrm{Hol}}
\newcommand{\M}{\mathcal{M}}
\newcommand{\A}{\mathcal{A}_\mathrm{flat}}
\newcommand{\G}{\mathcal{G}}
\newcommand{\g}{\mathfrak{g}}
\newcommand{\ol}[1]{\overline{#1}}
\newcommand{\Ad}[1]{\mathrm{Ad}_{#1}}
\newcommand{\LLL}{\ell}
\newcommand{\cover}{\iota}
\newcommand{\COVER}{I}
\newcommand{\deck}{\tau}
\newcommand{\DECK}{T}
\newcommand{\AAA}{\mathbb{A}}

\begin{document}

\title{Goldman flows on a nonorientable surface}
\author{David B. Klein}
\date{October 28, 2007}
\begin{abstract}
Given an embedded cylinder in an arbitrary surface, we give a gauge
theoretic definition of the associated Goldman flow, which is a circle
action on a dense open subset of the moduli space of equivalence
classes of flat $SU(2)$-connections over the surface.  A cylinder in a
compact nonorientable surface lifts to two cylinders in the orientable
double cover, and the \emph{composite flow} is the composition of one
of the associated flows with the inverse flow of the other.  Providing
explicit descriptions, we relate the flow on the moduli space of the
nonorientable surface with the composite flow on the moduli space of
the double cover.  We prove that the composite flow preserves a
certain Lagrangian submanifold.
\end{abstract}
\maketitle

\section{Introduction}

We generalize the Goldman flow of L.~Jeffrey and J.~Weitsman to the
moduli space of an arbitrary (possible nonorientable) surface: given an
embedded oriented cylinder in the surface, we define an associated
circle action on a dense open subset of the $SU(2)$ moduli space,
which coincides with action of L.~Jeffrey and J.~Weitsman when the
surface is compact and oriented.  Here, if $G$ is a Lie group then the
moduli space of a surface $S$ is the quotient $\M(S)=\A(S)/\G(S)$ of
the space of flat connections on the trivial principal $G$-bundle by the
group of gauge transformations.

We restrict our attention to a compact nonorientable surface.  An
embedded cylinder lifts to two disjoint cylinders in the orientable
double cover, and the two associated Goldman flows commute;
composing one of these flows with the inverse flow of the other
produces a circle action on the moduli space of the double cover,
which we shall call the \emph{composite flow}.  Using convenient
generators of the fundamental group of the nonorientable surface and
their preimages in the double cover, we give explicit descriptions both
of the Goldman flow on the moduli space of the surface, and of the
composite flow on the moduli space of the double cover.  The pullback
of the deck transformation induces an involution on the moduli space
of the double cover, and the fixed point set of this involution has been
shown by N.-K.~Ho to be a Lagrangian submanifold.  We prove that
the composite flow preserves this Lagrangian submanifold.  The
pullback of the covering map induces a map from the moduli space of
the surface to the moduli space of the double cover, and we prove that
the image of this map is also preserved by the composite flow.

This paper was inspired by the work of W.~Goldman.  The moduli
space of a surface $S$ may be identified with the space
$\mathrm{Hom}(\pi_1(S),G)/G$ of conjugacy classes of
homomorphisms from the fundamental group into the Lie group.  In
\cite{Goldman86}, starting with a simple closed curve in a compact
Riemann surface $S$ and a conjugation invariant function on a rather
general Lie group $G$, W.~Goldman defines an associated
$\R$-action on $\mathrm{Hom}(\pi_1(S),G)/G$.  After using the
invariant function and the simple closed curve to produce a function on
the symplectic space $\mathrm{Hom}(\pi_1(S),G)/G$, this $\R$-action
is the flow of the associated Hamiltonian vector field.  The typical
example of an invariant function on a Lie group is the trace function,
$\tr(g)$.  The Goldman flow for an arbitrary invariant function is
periodic when restricted to any one orbit, but the periods generally
differ from orbit to orbit.  In \cite{JeffreyWeitsman}, working with a
Riemann surface and the group $G=SU(2)$, L.~Jeffrey and
J.~Weitsman consider the invariant function $\Arccos((\tr g)/2)$.
Although the associated Goldman flow is defined only on a dense
open subset of the moduli space, it has \emph{single} period for each
orbit and thus defines a circle action on its domain.  This circle action
has been studied by various other authors; see, for instance,
\cite{Donaldson} or \cite{Tyurin}.

\subsection{Outline of the paper}

In Section~\ref{section1} we use the language of gauge theory to
define the Goldman flow associated to an embedded oriented cylinder
in an arbitrary surface; it is a circle action on an open dense subset of
the $SU(2)$ moduli space of the surface.  Specifically, we define two
$\R$-actions on the space of flat connections, one corresponding to
the left half of the cylinder and one corresponding to the right, both of
which cover the Goldman flow on the moduli space.  The key step in
defining these two $\R$-actions is to use Lemma 2.3 from
\cite{JeffreyWeitsman}, which says that a flat connection on the
surface can  be \emph{adapted} to the cylinder.  When the surface is
compact and oriented, our circle action coincides with the circle action
of L.~Jeffrey and J.~Weitsman given in \cite{JeffreyWeitsman}.

In Section~\ref{section2} we consider a compact nonorientable surface
with an embedded cylinder.  In this paper we assume that the cylinder
does not separate the surface into two pieces.  With minor
modifications, the technique used may treat the case where the
cylinder divides the surface in two.

In Sections~\ref{topology_surface} and~\ref{topology_cover} we
explore the topology of the nonorientable surface and its oriented
double cover.  We choose convenient generators of the fundamental
group of the surface, and view the surface as a polygon with edge
identifications.  Lifting these generators to the double cover, we view
the double cover as \emph{two} polygons with edge identifications.
The interior of each of the polygons representing the double cover is
mapped diffeomorphically by the covering map onto the interior of the
polygon representing the surface.

Section~\ref{identifications} is devoted to identifying the moduli spaces
of the surface and of the double cover with spaces that are much
easier to work with.  A well known construction allows us to identify the
moduli space of flat connections modulo \emph{based} gauge
transformations with a subset $\mathcal{R}$ of the direct sum of the
same number of copies of $G$ as there are generators; the set
$\mathcal{R}$ may also be identified with the space
$\mathrm{Hom}(\pi,G)$ of homomorphisms from the fundamental
group of the surface to $G$.  The Lie group acts on $\mathcal{R}$ by
conjugation on each factor, and the quotient is identified with the
moduli space of the nonorientable surface.  The preimage under the
covering map of the base point of the fundamental group is two points,
and lifting the generators of the fundamental group produces twice as
many curve in the double cover as there are generators; some of
these lifted curves are loops, and some are paths from one preimage
of the base point to the other.  Utilizing a construction of N.-K.~Ho that
appears in \cite{Ho}, we use the lifts of the generators to define a
subset $\widetilde{\mathcal{R}}$ of the direct sum of twice as many
copies of $G$ as there are generators, and we equip
$\widetilde{\mathcal{R}}$ with an action of $G\times G$ so that the
quotient is identified with the moduli space of the double cover.

The deck transformation of the double cover induces an involution on
the moduli space of the double cover, and the fixed point set of this
involution is shown in \cite{Ho} to be a Lagrangian submanifold.  In
Section~\ref{fixed_point} we adapt a lemma from \cite{Ho} that
expresses this fixed point set as a union of more manageable sets.

W.~Goldman defines the flow on $\mathrm{Hom}(\pi_1(S),G)/G$ for a
compact oriented surface $S$ as the projection of a certain flow on
$\mathrm{Hom}(\pi_1(S),G)$; see \cite{Goldman86},
\cite{Goldman97}, and \cite{Goldman04}.  Analogously, in
Section~\ref{goldman_flow} we give an explicit definition the Goldman
flow on the moduli space $\mathcal{R}/G$ of the nonorientable surface
by describing a lift of the flow to $\mathcal{R}$.  The cylinder in the
nonorientable surface lifts to two disjoint cylinders in the double cover,
and in Section~\ref{goldman_flow} we give explicit descriptions of the
two associated Goldman flows on the moduli space
$\widetilde{\mathcal{R}}/(G\times G)$ by defining two lifted flows on
$\widetilde{\mathcal{R}}$.  Composing one of these flows with the
inverse of the other gives the \emph{composite flow}.  Note that when
dealing with the composite flow on the moduli space of the double
cover, we can't use the analogy with homomorphisms from
fundamental group to $G$ because the base point in the nonorientable
surface forces us to consider \emph{two} base points in the double
cover.  The gauge theoretic definition of the Goldman flow developed
in Section~\ref{section1} easily handles the two base points.

In Section~\ref{proof} we prove that the composite flow on the moduli
space of the double cover preserves both the fixed point set of the
involution induced by the deck transformation, and the image of the
map between moduli spaces induced by the covering map.

\section{The $SU(2)$ Goldman flow}\label{section1}

\subsection{Notation and conventions}

Suppose $\Sigma$ is a real 2-dimensional smooth manifold.  The
surface $\Sigma$ may be nonorientable or noncompact.  Let
$\A(\Sigma)\subset\Omega^1(\Sigma)\otimes\mathfrak{su}(2)$ be the
space of flat connections on the trivial principal $SU(2)$-bundle over
$\Sigma$.  Our convention for principal bundles is to use the left action
induced by left multiplication in the Lie group; thus, for instance, the
curvature of a connection $A$ is
$\dee A-{\textstyle\frac{1}{2}}[A\wedge A]$.  Here's a useful formula: if
$\sigma\in C^\infty([0,1],\Sigma)$ is a curve and if $A$ is a connection
that takes values in an abelian Lie subalgebra of $\mathfrak{su}(2)$
when restricted to $\sigma$ then
\begin{equation}
\Hol_\sigma(A)=\exp\left(-{\textstyle \int_0^1}\sigma^*A\right).
\end{equation}
\noindent The gauge group is $\G(\Sigma)=C^\infty(\Sigma,SU(2))$,
and the moduli space is $\M(\Sigma)=\A(\Sigma)/\G(\Sigma)$.  Let
$G=SU(2)$ and let $\g=\mathfrak{su}(2)$.

\subsection{The cylinder}

Let $U=S^1\times[-1,1]\subset\Sigma$ be an embedded cylinder with
coordinates $(\theta,s)$ and orientation $\dee\theta\wedge\dee s$;
here, we are viewing the circle as $\R/2\pi\mathbb{Z}$.  If $\Sigma$ is
oriented then we assume that the embedding is orientation preserving.
Let $\gamma$ be the central curve $S^1\times\{0\}$ with orientation
$\dee\theta$ and base point $p=(0,0)$; see Figure~\ref{x1_cylinder}.

{\begin{figure}[h] 
\psfrag{p}{$p$}
\psfrag{c}{$\gamma$}
\psfrag{U}{$U$}
\psfrag{s}{$s$}
\psfrag{dt}{$\frac{\partial}{\partial\theta}$}
\psfrag{ds}{$\frac{\partial}{\partial s}$}
\psfrag{dtds}{$\text{orient}^{\underline{\text{n}}} \
 \dee\theta\wedge\dee s$}
\includegraphics{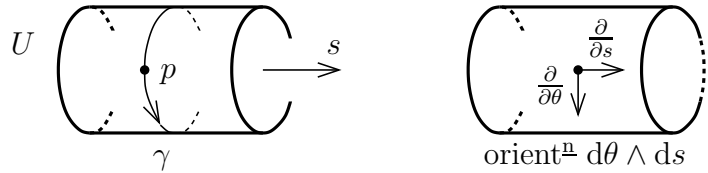}
\caption{\label{x1_cylinder}
The cylinder $U$, the curve $\gamma$, and the point $p$.}
\end{figure}}

\subsection{The function}

Let $G=SU(2)$ and consider the $\Ad{G}$-invariant inner product
$\langle\zeta,\eta\rangle={\textstyle -\frac{1}{2}}\tr(\zeta\eta)$ on the
Lie algebra $\g=\mathfrak{su}(2)$.  Let $f$ be the smooth $\R$-valued
conjugation invariant function on $G\smallsetminus\{\pm 1\}$ defined
as follows:
\begin{equation}
f(g)=\Arccos({\textstyle \frac{1}{2}}\tr(g)).
\end{equation}
\noindent Define a \emph{logarithm} map $\LLL$ on
$G\smallsetminus\{\pm 1\}$ by requiring that $\exp(\LLL(g))=g$ and
$\langle\LLL(g),\LLL(g)\rangle<\pi^2$.  The \emph{variation} of $f$ with
respect to $\langle\cdot,\cdot\rangle$ is the $\g$-valued function $F$
on $G\smallsetminus\{\pm 1\}$ defined by setting
$\langle F(g),\xi\rangle=\dee f_g(\xi^\mathrm{L})
=\dee f_g(\xi^\mathrm{R})$, for all $\xi\in\g$.  Explicitly, if
$g\in G\smallsetminus\{\pm 1\}$, let $\alpha=f(g)\in(0,\pi)$ and write
$g=x\left(\begin{smallmatrix}e^{i\alpha}&\\&e^{-i\alpha}
\end{smallmatrix}\right)x^{-1}$ for some $x\in G$; then
\begin{equation}\label{logarithm}
\LLL(g)=x\left(\begin{smallmatrix}i\alpha&\\&-i\alpha
\end{smallmatrix}\right)x^{-1},
\end{equation}
\noindent and
\begin{equation}\label{variation}
F(g)={\textstyle \frac{1}{\sqrt{4-(\tr g)^2}}}(g-g^{-1})
={\textstyle \frac{1}{\sqrt{\langle\LLL(g),\LLL(g)\rangle}}}\LLL(g)
=x\left(\begin{smallmatrix}i&\\&-i\end{smallmatrix}\right)x^{-1}.
\end{equation}
\noindent We shall eventually need the following fact:
if $g\in G\smallsetminus\{\pm 1\}$, $x\in G$, and $xgx^{-1}=g$ then
\begin{equation}\label{variation_commute}
\Ad{x}F(g)=F(g).
\end{equation}

\subsection{The Goldman flow}

Let $\mathcal{S}_\gamma=\{A\in\A(\Sigma) \ : \
\Hol_\gamma A\neq\pm 1\}$, and let $\M_\gamma
=\mathcal{S}_\gamma/\G(\Sigma)\subset\M(\Sigma)$.  Use the curve
$\gamma$ to define an $\R$-valued function $f_\gamma$ on
$\mathcal{S}_\gamma$,
\begin{equation}
f_\gamma(A)=f(\Hol_\gamma A).
\end{equation}
\noindent The Goldman flow associated to the cylinder $U$, which we
define in the proof of the following theorem, is a periodic $\R$-action
$\{\Xi_t\}_{t\in\R}$ on the dense open subset $\M_\gamma$ of the
moduli space.

\begin{theorem}\label{theorem_flow}
There are $\R$-actions $\{\Xi_t^+\}_{t\in\R}$ and $\{\Xi_t^-\}_{t\in\R}$
on $\mathcal{S}_\gamma$ satisfying the following conditions:
\begin{itemize}
\item[(i)] The $\R$-actions $\Xi_t^\pm$ have ``support'' in $U$ in the
following sense: $\Xi_t^-(A)=A$ outside of some compact subset of
$S^1\times(-1,0)\subset U$, and $\Xi_t^+(A)=A$ outside of some
compact subset of $S^1\times(0,1)\subset U$.
\item[(ii)] If $\dee(f_\gamma)_A$ is the tangent map of $f_\gamma$ at
$A$ then
\begin{equation}
\dee(f_\gamma)_A(B)=\int_U
\left\langle({\textstyle \dd{t}{0}}\Xi_t^\pm A)\wedge B\right\rangle,
\end{equation}
\noindent for $B\in T_A\mathcal{S}_\gamma=T_A\A(\Sigma)$.
\item[(iii)] $\Xi_t^\pm$ are $\G$ equivariant: if $A\in\A$ and $\psi\in\G$
then $\Xi_t^\pm(\psi.A)=\psi.(\Xi_t^\pm(A))$.
\item[(iv)] If $A\in\mathcal{S}_\gamma$ and $t\in\R$ then there exists
$\psi\in\G$ such that $\psi.\Xi_t^-(A)=\Xi_t^+(A)$.
\end{itemize}
The $\R$-actions $\Xi_t^+$ and $\Xi_t^-$ on $\mathcal{S}_\gamma$
thus define a common $\R$-action $\{\Xi_t\}_{t\in\R}$ on
$\M_\gamma$.  The action $\Xi_t$ is periodic, with period $\pi$ if
$\Sigma\smallsetminus\gamma$ is disconnected and period $2\pi$ if
$\Sigma\smallsetminus\gamma$ is connected.  The $S^1$-action on
$\M_\gamma$ defined by $\Xi_t$ is called the Goldman flow
associated to the cylinder $U$.
\end{theorem}

\begin{proof}
We shall construct the maps $\Xi_t^\pm$, and leave the proofs of the
above statements as an exercise for the reader.  Let
$\eta_-(s),\eta_+(s)\in C^\infty(\R)$ be smooth bump functions with
compact support in either $(-1,0)$ or $(0,1)$ that have integral $1$;
see Figure~\ref{x2_bump_function}.  We use the following variation of
Lemma 2.3 from \cite{JeffreyWeitsman}: if $A\in\mathcal{S}_\gamma$
then $\Hol_\gamma A=\exp(\LLL(\Hol_\gamma A))$, and there is a
unique gauge transformation $u\in\G(U)$ on $U$ with $u(p)=1$ such
that
\begin{equation}\label{adapted_equation}
u.(A|_U)=-\frac{\dee\theta}{2\pi}\otimes\LLL(\Hol_\gamma A).
\end{equation}
\noindent Here, we say that $u.(A|_U)$ is \emph{adapted} to the
cylinder $U$, or that $u$ adapts $A$ to $U$.  For $t\in\R$ define
\begin{equation}
\Xi_t^\pm(A)=A+\underbrace{
\Ad{u^{-1}}(\eta_\pm(s)\dee s\otimes tF(\Hol_\gamma A))}_{\in\A(U)},
\end{equation}
\noindent where the second term
$\Ad{u^{-1}}(\eta_\pm(s)\dee s\otimes tF(\Hol_\gamma A))$ is
supported on a compact subset of $S^1\times(-1,1)\subset U$ and
therefore extends by 0 to a flat connection on $\Sigma$.

To prove~(ii), use the infinitesimal version of
equation~(\ref{adapted_equation}): if $A\in\mathcal{S}_\gamma$ and
$B\in T_A\A(\Sigma)$ then there exists a Lie algebra-valued function
$\varphi\in C^\infty(U)\otimes\g$ on the cylinder such that
\begin{equation}
(B|_U)+\dee_A\varphi=-\frac{\dee\theta}{2\pi}\otimes\zeta,
\end{equation}
\noindent for some $\zeta\in\g$ satisfying
$[\zeta,F(\Hol_\gamma A)]=0$.
\end{proof}

{\begin{figure}[h] 
\psfrag{eta-}{${\scriptstyle \eta_-(s)}$}
\psfrag{eta+}{${\scriptstyle \eta_+(s)}$}
\psfrag{s}{${\scriptstyle s}$}
\psfrag{+1/4}{${\scriptscriptstyle \frac{1}{4}}$}
\psfrag{-1/4}{${\scriptscriptstyle -\frac{1}{4}}$}
\psfrag{+1/2}{${\scriptscriptstyle \frac{1}{2}}$}
\psfrag{-1/2}{${\scriptscriptstyle -\frac{1}{2}}$}
\psfrag{0}{${\scriptstyle 0}$}
\includegraphics{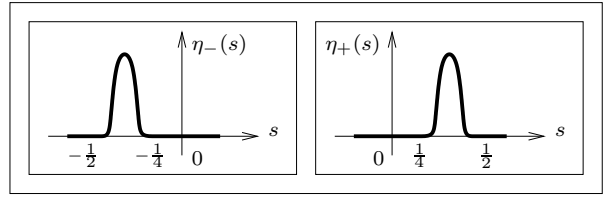}
\caption{\label{x2_bump_function}
The bump functions $\eta_\pm(s)\in C^\infty(\R)$, with area
$\int_{-\infty}^\infty\eta_\pm(s)\dee s=1$.}
\end{figure}}

The function $f_\gamma$ is $\G$-invariant and thus defines a function
on $\M_\gamma$, (which we still call $f_\gamma$).  When $\Sigma$
is oriented and compact, there is a well-known symplectic structure on
$\M(\Sigma)$ given by
\begin{equation}\label{symplectic_form}
\omega_{[A]}([B],[C])=\int_\Sigma\langle B\wedge C\rangle,
\quad \text{for $B,C\in T_A\A$};
\end{equation}
\noindent see, for instance, \cite{AtiyahBott} or \cite{McDuffSalamon}.
The following corollary is an immediate consequence of item~(ii) in
Theorem~\ref{theorem_flow}.

\begin{corollary}
\noindent If $\Sigma$ is oriented and compact then the Goldman flow
${\Xi_t}$ is the flow of the Hamiltonian vector field on $\M_\gamma$
with Hamilton function $f_\gamma$.
\end{corollary}

\subsection{Holonomy}

We shall eventually choose generators of the fundamental group, and
identify the moduli space with a subset $\mathcal{R}/G$ of
$(G\times\cdots\times G)/G$ by taking holonomies along the
generators.  The following theorem, which describes how the
holonomy along certain types of curves is effected by the Goldman
flow, will allow us to work with the Goldman flow on $\mathcal{R}/G$.

\begin{theorem}\label{theorem_holonomy}
Let $\sigma:[0,1]\rightarrow\Sigma$ be a curve in $\Sigma$, with one
endpoint at $p$ and the other endpoint either at $p$ or in
$\Sigma\smallsetminus U$, and suppose that $\sigma$ does not
otherwise intersect $\gamma$.  If $A\in\mathcal{S}_\gamma$, let
\begin{equation}
\zeta_t=\zeta_t(A)=\exp(tF(\Hol_\gamma A)),
\end{equation}
\noindent where $F$ is defined in~(\ref{variation}); then the holonomy
of $\Xi_t^\pm(A)$ along such a curve $\sigma$ is given in the table
appearing in Figure~\ref{x3_table}.
\end{theorem}

{\begin{figure}[t] 
\psfrag{Xi-}{$\Hol_\sigma(\Xi_t^-(A))$}
\psfrag{Xi+}{$\Hol_\sigma(\Xi_t^+(A))$}
\psfrag{h}{$\Hol_\sigma A$}
\psfrag{ah}{$(\zeta_t)(\Hol_\sigma A)(\zeta_t^{-1})$}
\psfrag{AH}{$(\zeta_t^{-1})(\Hol_\sigma A)(\zeta_t)$}
\psfrag{ht}{$(\Hol_\sigma A)(\zeta_t)$}
\psfrag{HT}{$(\Hol_\sigma A)(\zeta_t^{-1})$}
\psfrag{th}{$(\zeta_t)(\Hol_\sigma A)$}
\psfrag{TH}{$(\zeta_t^{-1})(\Hol_\sigma A)$}
\psfrag{sigma}{$\sigma$}
\psfrag{p}{$p$}
\psfrag{c}{$\gamma$}
\psfrag{U}{$U$}
\psfrag{s}{$s$}
\psfrag{dt}{${\scriptstyle \frac{\partial}{\partial\theta}}$}
\psfrag{ds}{${\scriptstyle \frac{\partial}{\partial s}}$}
\psfrag{dtds}{$\text{orient}^{\underline{\text{n}}} \
 \dee\theta\wedge\dee s$}
\includegraphics{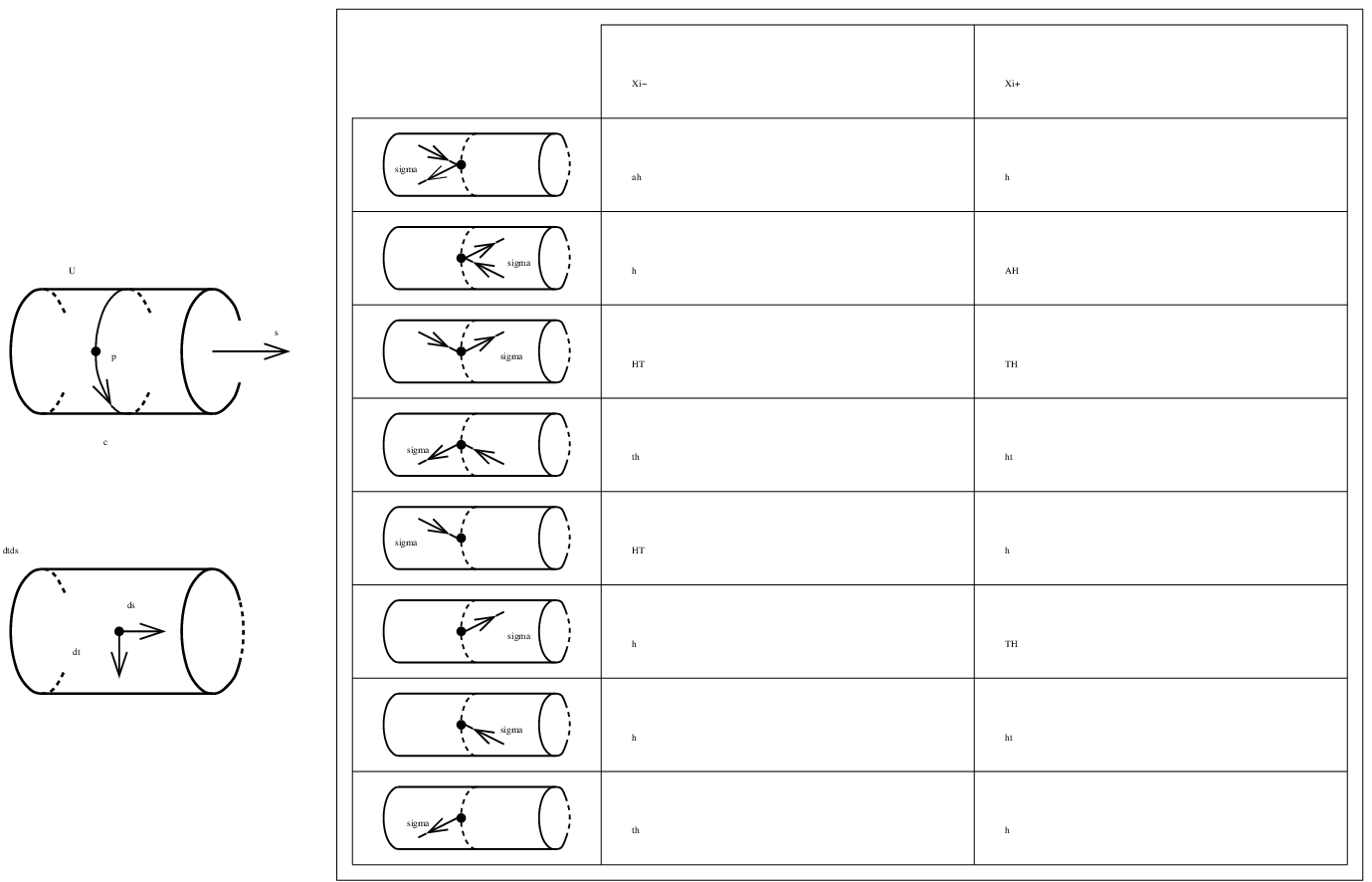}
\caption{\label{x3_table}
The holonomy of $\Xi_t^\pm(A)$ along $\sigma$; see
Theorem~\ref{theorem_holonomy} and
Remark~\ref{remark_holonomy}.}
\end{figure}}

\begin{remark}\label{remark_holonomy}
Since the holonomy of a flat connection along a curve doesn't change
if the curve is deformed by a homotopy that fixes its endpoints, in
Theorem~\ref{theorem_holonomy} we only need to consider how the
curve $\sigma$ behaves near its endpoints relative to the point $p$,
the curve $\gamma$, and the cylinder $U$.  The first column of the
table in Figure~\ref{x3_table} shows the eight possible behaviours.
\end{remark}

{\begin{figure}[b] 
\psfrag{Upsilon-}{${\scriptstyle \Upsilon_-(s)}$}
\psfrag{Upsilon+}{${\scriptstyle \Upsilon_+(s)}$}
\psfrag{s}{${\scriptstyle s}$}
\psfrag{+1/2}{${\scriptscriptstyle \frac{1}{2}}$}
\psfrag{-1/2}{${\scriptscriptstyle -\frac{1}{2}}$}
\psfrag{1}{${\scriptstyle 1}$}
\psfrag{-1}{${\scriptstyle -1}$}
\psfrag{0}{${\scriptstyle 0}$}
\includegraphics{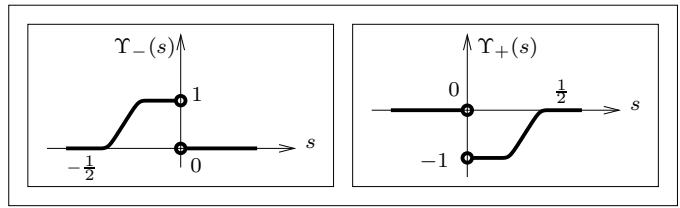}
\caption{\label{x4_jump_function}
The jump functions
$\Upsilon_\pm(s)\in C^\infty(\R\smallsetminus\{0\})$, with
$\frac{\dee}{\dee s}\Upsilon_\pm(s)=\eta_\pm(s)$.}
\end{figure}}

\begin{proof}[Proof of Theorem~\ref{theorem_holonomy}]
Fix $A\in\mathcal{S}_\gamma$, and let $u.(A|_U)$ be adapted to $U$
as in equation~(\ref{adapted_equation}).  Let $\Upsilon_-(s)$ and
$\Upsilon_+(s)$ be the compactly supported jump functions on
$\R\smallsetminus\{0\}$ obtained by integrating the bump functions
$\eta_-(s)$ and $\eta_+(s)$; see Figure~\ref{x4_jump_function}.  The
two gauge transformations
\begin{equation}
\psi_t^\pm=u^{-1}\exp(t\Upsilon_\pm(s)\otimes F(\Hol_\gamma A))u
\end{equation}
\noindent on $U\smallsetminus\gamma$ extend by 1 to gauge
transformations $\psi_t^\pm\in\G(\Sigma\smallsetminus\gamma)$,
which satisfy
\begin{equation}
\psi_t^\pm.(A|_{\Sigma\smallsetminus\gamma})
=(\Xi^\pm_t A)|_{\Sigma\smallsetminus\gamma}.
\end{equation}
\noindent To complete the proof, use the fact that if $\psi$ is a gauge
transformation on $\Sigma$ then the holonomies of $A$ and $\psi.A$
along $\sigma$ are related by
\begin{equation}
\Hol_\sigma(\psi.A)
=(\psi(\sigma_0))(\Hol_\sigma A)(\psi(\sigma_1))^{-1},
\end{equation}
\noindent and consider the limits of $\psi_t^\pm(p_0)$ as $p_0$ tends
to $p$ in $U\smallsetminus\gamma$ from the left and from the right.
\end{proof}

\section{The flow for a nonorientable surface}\label{section2}

For the remainder of this paper, let $\Sigma$ be a compact connected
nonorientable surface with an embedded oriented cylinder $U$.  As in
Figure~\ref{x1_cylinder}, view $U$ as a tubular neighbourhood of the
oriented curve $\gamma$ with base point $p$.

We shall assume that the surface $\Sigma\smallsetminus\gamma$ is
connected; the method used easily adapts to the case where
$\Sigma\smallsetminus\gamma$ is disconnected, and the interested
reader may find this a worthwhile exercise.  Note that $\Sigma$ cannot
be $\R P^2$, because a simple closed curve in $\R P^2$ either is
contractible (and bounds a disk) or else has no orientable tubular
neighbourhood.

{\begin{figure}[h] 
\psfrag{c}{${\scriptstyle \gamma}$}
\psfrag{b}{${\scriptstyle \beta}$}
\psfrag{p}{${\scriptscriptstyle p}$}
\psfrag{U}{$U$}
\psfrag{d}{$\delta$}
\psfrag{subset}{{\LARGE $\subseteq$}}
\psfrag{orientable}{$\Sigma_+$ orientable}
\psfrag{nonorientable}{$\Sigma_+$ nonorientable}
\psfrag{ddd1}{${\scriptstyle \delta \ \sim \
 \beta^{-1}\gamma\beta\gamma^{-1}}$}
\psfrag{ddd2}{${\scriptstyle \delta \ \sim \
 \beta^{-1}\gamma^{-1}\beta\gamma^{-1}}$}
\includegraphics{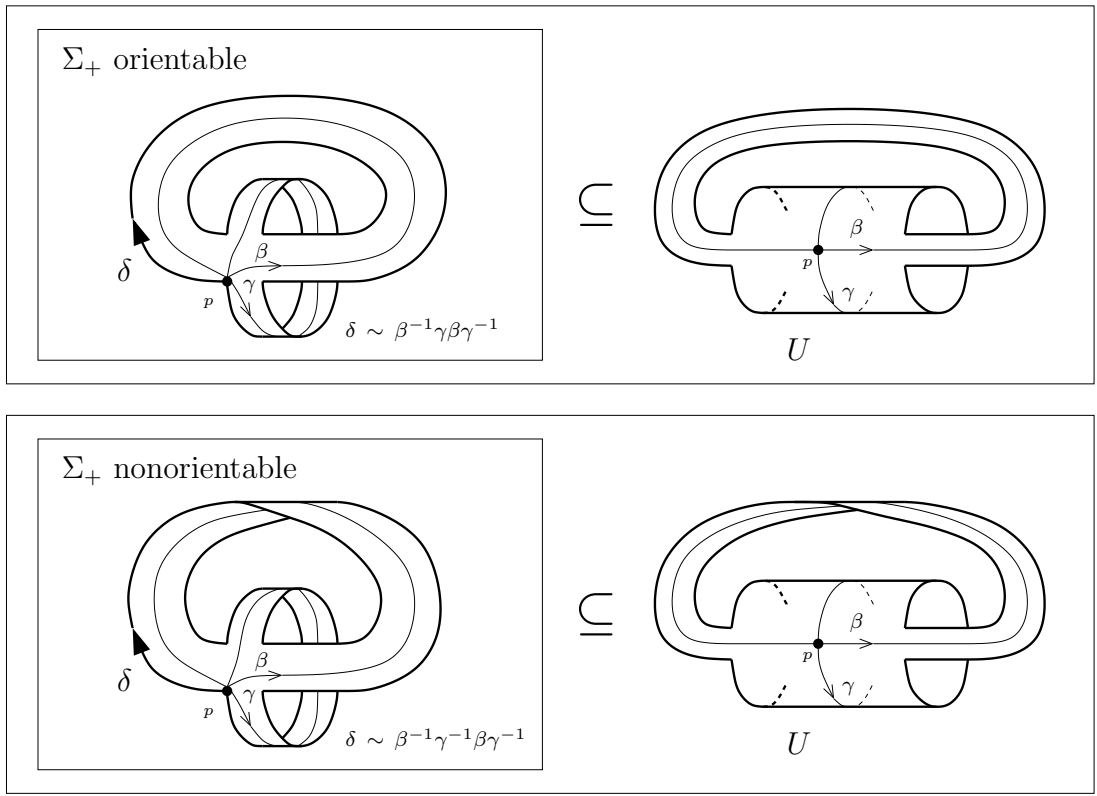}
\caption{\label{x5_surface_plus}
The surface $\Sigma_+$ is either orientable or nonorientable.}
\end{figure}}

\subsection{The topology of $\Sigma$}\label{topology_surface}

Since $\Sigma\smallsetminus\gamma$ is connected, we may choose
a simple closed curve $\beta$ that intersects $\gamma$ exactly once,
transitively at $p$.  Recall that the cylinder $U$ has coordinates
$(s,\theta)$, and orient $\beta$ so that it leaves $p$ in the positive $s$
direction and approaches $p$ from the negative $s$ direction.  Two
things may occur: either $\beta$ possesses an orientable tubular
neighbourhood in $\Sigma$, or it does not.  A sketch of $U$ and a
tubular neighbourhood of $\beta$ appears in the right half of
Figure~\ref{x5_surface_plus}.  Let $\Sigma_+$ be the surface with
boundary pictured in the left half of Figure~\ref{x5_surface_plus}.  The
surface $\Sigma_+$ is either orientable or nonorientable, and in either
case it has a single boundary component, which we call $\delta$.
Viewed as an element of $\pi_1(\Sigma_+,p)$,
\begin{equation}
\delta=\begin{cases}
 \beta^{-1}\gamma\beta\gamma^{-1}
 & \text{for $\Sigma_+$ orientable,} \\
 \beta^{-1}\gamma^{-1}\beta\gamma^{-1}
 & \text{for $\Sigma_+$ nonorientable.} \\ \end{cases}
\end{equation}

Let $\Sigma_-$ be the surface obtained by deleting the interior of
$\Sigma_+$ from $\Sigma$.  The surface $\Sigma_-$ has a single
boundary component, the curve $\delta$, and gluing together
$\Sigma_+$ and $\Sigma_-$ along $\delta$ recovers $\Sigma$.  The
surface $\Sigma_-$ is either nonorientable or orientable; since our
original surface $\Sigma$ is nonorientable, the only restriction is that
$\Sigma_+$ and $\Sigma_-$ cannot both be orientable.  We must thus
consider the following three cases:
\begin{equation}
{\renewcommand{\arraystretch}{1.25}\begin{array}{|rl|}\hline
 \text{case (i):}
 & \text{$\Sigma_-$ nonorientable, $\Sigma_+$ orientable} \\
 \text{case (ii):}
 & \text{$\Sigma_-$ nonorientable, $\Sigma_+$ nonorientable} \\
 \text{case (iii):}
 & \text{$\Sigma_-$ orientable, $\Sigma_+$ nonorientable} \\
 \hline\end{array}}
\end{equation}
\noindent If $\Sigma_-$ is nonorientable then it is diffeomorphic to a
disk with $k>0$ M\"obius strips attached, and we choose generators
$\alpha_1,\dots\alpha_k$ of the rank $k$ free group
$\pi_1(\Sigma_-,p)$ that satisfy
\begin{equation}
\delta=\alpha_1^{\phantom{1}2}\cdots\alpha_k^{\phantom{k}2};
\end{equation}
\noindent this situation is pictured in the bottom half of
Figure~\ref{x6_minus_cover}.  If $\Sigma_-$ is orientable then it is
diffeomorphic to a disk with $k\ge 0$ handles attached, and we
choose generators $\alpha_1,\dots\alpha_{2k}$ of the rank $2k$ free
group $\pi_1(\Sigma_-,p)$ that satisfy
\begin{equation}
\delta=[\alpha_1,\alpha_2]\cdots[\alpha_{2k-1},\alpha_{2k}].
\end{equation}
\noindent The fundamental group of $\Sigma$ is as follows:
\begin{equation}\label{fundamental_group}\begin{aligned}
\text{case (i),}\quad
& \pi_1(\Sigma,p)=
 \left\langle\gamma,\beta,\alpha_1,\dots,\alpha_k \ | \
 \beta^{-1}\gamma\beta\gamma^{-1}=
 \alpha_1^{\phantom{1}2}\cdots\alpha_k^{\phantom{k}2}\right\rangle; \\
\text{case (ii),}\quad
& \pi_1(\Sigma,p)=
 \left\langle\gamma,\beta,\alpha_1,\dots,\alpha_k \ | \
 \beta^{-1}\gamma^{-1}\beta\gamma^{-1}=
 \alpha_1^{\phantom{1}2}\cdots\alpha_k^{\phantom{k}2}\right\rangle; \\
\text{case (iii),}\quad
& \pi_1(\Sigma,p)=
 \left\langle\gamma,\beta,\alpha_1,\dots,\alpha_{2k} \ | \
 \beta^{-1}\gamma^{-1}\beta\gamma^{-1}
 =[\alpha_1,\alpha_2]\cdots[\alpha_{2k-1},\alpha_{2k}]\right\rangle.
\end{aligned}\end{equation}

\begin{remark}\label{generators_remark}
The number of generators appearing in our presentation of
$\pi_1(\Sigma,p)$ can be expressed in terms of the Euler
characteristic of $\Sigma$.  If $\Sigma_-$ is nonorientable then
$\chi(\Sigma_-)=1-k$, and if $\Sigma_-$ is orientable then
$\chi(\Sigma_-)=1-2k$.  In all three cases, $\chi(\Sigma_+)=-1$ and
$\chi(\Sigma)=\chi(\Sigma_+)+\chi(\Sigma_-)$.  Therefore, the number
of generators in equation~(\ref{fundamental_group}) is
\begin{equation}
2-\chi(\Sigma)=\begin{cases}
 2+k, & \text{cases (i) and (ii),} \\
 2+2k, & \text{case (iii).}\end{cases}
\end{equation}
\end{remark}

{\begin{figure}[t] 
\psfrag{c}{${\scriptstyle \gamma}$}
\psfrag{C}{${\scriptstyle \Gamma}$}
\psfrag{0C}{${\scriptstyle \ol{\Gamma}}$}
\psfrag{b}{${\scriptstyle \beta}$}
\psfrag{B}{${\scriptstyle B}$}
\psfrag{0B}{${\scriptstyle \ol{B}}$}
\psfrag{a1}{${\scriptstyle \alpha_1}$}
\psfrag{A1}{${\scriptstyle A_1}$}
\psfrag{0A1}{${\scriptstyle \ol{A}_1}$}
\psfrag{ak}{${\scriptstyle \alpha_k}$}
\psfrag{Ak}{${\scriptstyle A_k}$}
\psfrag{0Ak}{${\scriptstyle \ol{A}_k}$}
\psfrag{p}{${\scriptscriptstyle p}$}
\psfrag{P}{${\scriptscriptstyle P}$}
\psfrag{0P}{${\scriptscriptstyle \ol{P}}$}
\psfrag{d}{$\delta$}
\psfrag{D}{$\Delta$}
\psfrag{0D}{$\ol{\Delta}$}
\psfrag{S-}{{\large $\Sigma_-$}}
\psfrag{0S-}{{\large $\widetilde{\Sigma}_-$}}
\includegraphics{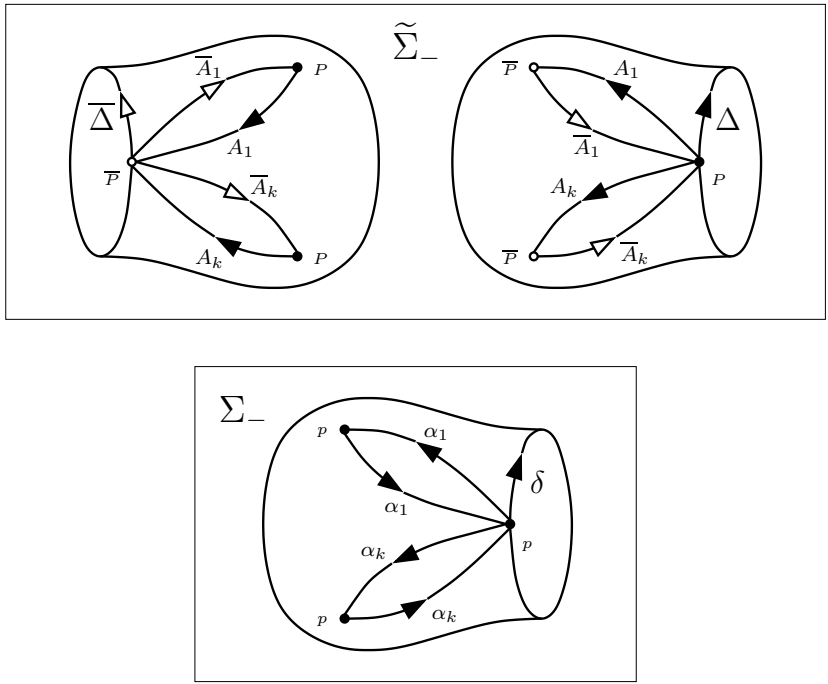}
\caption{\label{x6_minus_cover}
The double cover $\widetilde{\Sigma}_-$ for $\Sigma_-$
nonorientable.}
\end{figure}}

\subsection{The double cover $\widetilde{\Sigma}$}
\label{topology_cover}

Choose an orientation on the orientable double cover
$\widetilde{\Sigma}$ of $\Sigma$.  The preimage in
$\widetilde{\Sigma}$ of the cylinder $U$ under the covering map is two
cylinders $V$ and $\ol{V}$, both of which we orient using the fixed
orientation on $\widetilde{\Sigma}$, that are distinguished by requiring
the covering map to be orientation preserving on $V$ and orientation
reversing on $\ol{V}$.  The base point $p$ lifts to two points $P\in V$
and $\ol{P}\in\ol{V}$, and the oriented curve $\gamma$ lifts to two
oriented curves $\Gamma\subset V$ and $\ol{\Gamma}\subset\ol{V}$.

Let $B$ and $\ol{B}$ be the lifts of $\beta$, and for each $j$ let $A_j$
and $\ol{A}_j$ be the lifts of $\alpha_j$; here, the lifts $B$ and $A_j$
start at $P$, and the lifts $\ol{B}$ and $\ol{A}_j$ start at $\ol{P}$.  If a
loop based at $p$ in $\Sigma$ possesses an orientable tubular
neighbourhood then it lifts to two loops in $\widetilde{\Sigma}$, one
based at $P$ and one based at $\ol{P}$.  If, on the other hand, a loop
in $\Sigma$ possesses no orientable tubular neighbourhood then it
lifts to two paths, one from $P$ to $\ol{P}$ and one from $\ol{P}$ to
$P$.

Let the subsets $\widetilde{\Sigma}_-$ and $\widetilde{\Sigma}_+$ of
$\widetilde{\Sigma}$ be the double covers of $\Sigma_-$ and
$\Sigma_+$, respectively, and let $\Delta$ and $\ol{\Delta}$ be the lifts
of the curve $\delta$ with $P\in\Delta$ and $\ol{P}\in\ol{\Delta}$.  The
boundary of both $\widetilde{\Sigma}_-$ and $\widetilde{\Sigma}_+$ is
the (disjoint) union of $\Delta$ and $\ol{\Delta}$, and
$\widetilde{\Sigma}$ is obtained by gluing together
$\widetilde{\Sigma}_-$ and $\widetilde{\Sigma}_+$ along this common
boundary.

If $\Sigma_+$ is orientable then its double cover
$\widetilde{\Sigma}_+$ is the disjoint union of the following two
components: a copy of $\Sigma_+$ with $\gamma$, $\beta$, $\delta$,
and $p$ relabeled as $\Gamma$, $B$, $\Delta$, and $P$; a copy of
the mirror image of $\Sigma_+$ with $\gamma$, $\beta$, $\delta$, and
$p$ relabeled as $\ol{\Gamma}$, $\ol{B}$, $\ol{\Delta}$, and $\ol{P}$.
When $\Sigma_+$ is nonorientable, its double cover
$\widetilde{\Sigma}_+$ is pictured in Figure~\ref{x7_plus_cover},
where the deck transformation is a $180^\circ$ rotary-reflection (i.e.\ a
$180^\circ$ rotation followed by a reflection that fixes the axis of
rotation).  The case where $\Sigma_-$ is nonorientable is shown in the
bottom half of Figure~\ref{x6_minus_cover}; its double cover
$\widetilde{\Sigma}_-$ appears in the top half of the figure, where the
deck transformation is the obvious reflection.  If $\Sigma_-$ is
orientable then its double cover $\widetilde{\Sigma}_-$ is the disjoint
union of the following two components: a copy of $\Sigma_-$ with
$\delta$ and $p$ relabeled as $\Delta$ and $P$, and with each
$\alpha_j$ relabeled as $A_j$; a copy of the mirror image of
$\Sigma_-$ with $\delta$ and $p$ relabeled as $\ol{\Delta}$ and
$\ol{P}$, and with each $\alpha_j$ relabeled as $\ol{A}_j$.

{\begin{figure}[t] 
\psfrag{C}{${\scriptstyle \Gamma}$}
\psfrag{0C}{${\scriptstyle \ol{\Gamma}}$}
\psfrag{B}{${\scriptstyle B}$}
\psfrag{0B}{${\scriptstyle \ol{B}}$}
\psfrag{P}{${\scriptscriptstyle P}$}
\psfrag{0P}{${\scriptscriptstyle \ol{P}}$}
\psfrag{D}{$\Delta$}
\psfrag{0D}{$\ol{\Delta}$}
\psfrag{DDD}{${\scriptstyle \Delta \ \sim \
 \ol{B}^{-1}\ol{\Gamma}^{-1}\ol{B}\Gamma^{-1}}$}
\psfrag{0DDD}{${\scriptstyle \ol{\Delta} \ \sim \
 B^{-1}\Gamma^{-1}B\ol{\Gamma}^{-1}}$}
\includegraphics{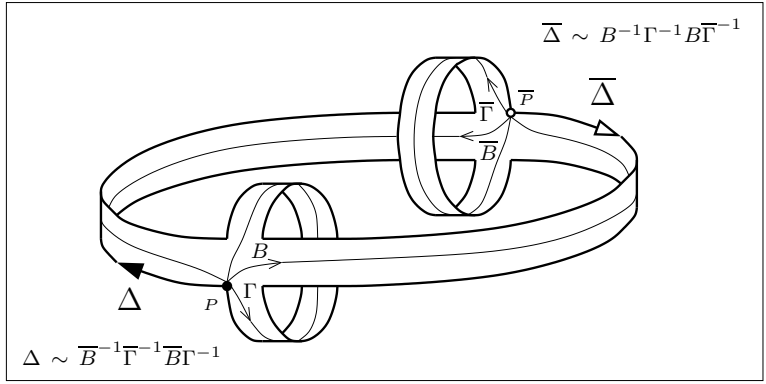}
\caption{\label{x7_plus_cover}
The double cover $\widetilde{\Sigma}_+$ for $\Sigma_+$
nonorientable.}
\end{figure}}

\begin{remark}\label{polygon_remark}
Since the presentation of $\pi_1(\Sigma,p)$ given in
equation~(\ref{fundamental_group}) has a \emph{single} relation,
which results from setting $\delta\in\pi_1(\Sigma_-,p)$ equal to
$\delta\in\pi_1(\Sigma_+,p)$, the surface $\Sigma$ can be realized as
a polygon with the following (clockwise) boundary:
\begin{align*}
\text{case (i),}\quad
& \gamma\beta^{-1}\gamma^{-1}\beta
 \alpha_1^{\phantom{1}2}\cdots\alpha_k^{\phantom{k}2}; \\
\text{case (ii),}\quad
& \gamma\beta^{-1}\gamma\beta
 \alpha_1^{\phantom{1}2}\cdots\alpha_k^{\phantom{k}2}; \\
\text{case (iii),}\quad
& \gamma\beta^{-1}\gamma\beta
 [\alpha_1,\alpha_2]\cdots[\alpha_{2k-1},\alpha_{2k}].
\end{align*}
\noindent We can repeat this process in $\widetilde{\Sigma}$.  View
$\Delta$ as an element of both $\pi_1(\widetilde{\Sigma}_-,P)$ and
$\pi_1(\widetilde{\Sigma}_+,P)$, and and view $\ol{\Delta}$ as an
element of both $\pi_1(\widetilde{\Sigma}_-,\ol{P})$ and
$\pi_1(\widetilde{\Sigma}_+,\ol{P})$.  The surface $\widetilde{\Sigma}$
can be realized as two polygons with the following boundaries:
\begin{equation}\label{polygons}\begin{aligned}
\text{case(i),}\quad
& \begin{cases}
 \text{clockwise}
 & \Gamma B^{-1}\Gamma^{-1}BA_1\ol{A}_1\cdots A_k\ol{A}_k, \\
 \text{counterclockwise}
 & \ol{\Gamma} \, \ol{B}^{-1}\ol{\Gamma}^{-1}\ol{B} \,
  \ol{A}_1 A_1\cdots\ol{A}_k A_k; \end{cases} \\
\text{case(ii),}\quad
& \begin{cases}
 \text{clockwise}
 & \Gamma\ol{B}^{-1}\ol{\Gamma} \, \ol{B}
  A_1\ol{A}_1\cdots A_k\ol{A}_k, \\
 \text{counterclockwise}
 & \ol{\Gamma}B^{-1}\Gamma B
  \ol{A}_1 A_1\cdots\ol{A}_k A_k; \end{cases} \\
\text{case(iii),}\quad
& \begin{cases}
 \text{clockwise}
 & \Gamma\ol{B}^{-1}\ol{\Gamma} \, \ol{B}
  [A_1,A_2]\cdots[A_{2k-1},A_{2k}], \\
 \text{counterclockwise}
 & \ol{\Gamma}B^{-1}\Gamma B
  [\ol{A}_1,\ol{A}_2]\cdots[\ol{A}_{2k-1},\ol{A}_{2k}]. \end{cases}
\end{aligned}\end{equation}
\end{remark}

\subsection{The identifications}\label{identifications}

For the rest of this paragraph let $K=-\chi(\Sigma)$, and recall from
Remark~\ref{generators_remark} that the presentation of
$\pi_1(\Sigma,p)$ appearing in equation~(\ref{fundamental_group})
has $2+K$ generators.  Introduce the following notation:
\begin{align*}
(c,b,\AAA) & =(c,b,a_1,\dots,a_K)\in G^{2-\chi(\Sigma)}, \\
(c,b,\AAA,\ol{c},\ol{b},\ol{\AAA})
& =(c,b,a_1,\dots,a_K,\ol{c},\ol{b},\ol{a}_1,\dots,\ol{a}_K)
 \in G^{2(2-\chi(\Sigma))}.
\end{align*}
\noindent We can left and right multiply as follows: if $x,y\in G$ then,
for instance,
\begin{equation*}
x\AAA y=(xa_1 y,\dots,xa_K y).
\end{equation*}

Consider the following subset $\mathcal{R}$ of $G^{2-\chi(\Sigma)}$:
\begin{equation}\begin{aligned}
\text{case(i),}\quad & \mathcal{R}=\left\{(c,b,\AAA)\in G^{2+k} \ \ | \ \
 b^{-1}cbc^{-1}=
 a_1^{\phantom{1}2}\cdots a_k^{\phantom{k}2}\right\}; \\
\text{case(ii),}\quad & \mathcal{R}=\left\{(c,b,\AAA)\in G^{2+k} \ \ | \ \
 b^{-1}c^{-1}bc^{-1}=
 a_1^{\phantom{1}2}\cdots a_k^{\phantom{k}2}\right\}; \\
\text{case(iii),}\quad & \mathcal{R}=\left\{(c,b,\AAA)\in G^{2+2k} \ \ | \ \
 b^{-1}c^{-1}bc^{-1}=
 [a_1,a_2]\cdots[a_{2k-1},a_{2k}]\right\}.
\end{aligned}\end{equation}
\noindent  Let $G$ act on $\mathcal{R}$, and on $G^{2-\chi(\Sigma)}$,
as follows:
\begin{equation}
g.(c,b,\AAA)=(gcg^{-1},gbg^{-1},g\AAA g^{-1}).
\end{equation}
\noindent Denote by $\G(\Sigma,p)$ the set of gauge transformations
on $\Sigma$ that evaluate to 1 at $p$.  We identify $\mathcal{R}$ with
$\A(\Sigma)/\G(\Sigma,p)$ by taking holonomies around the
generators of $\pi_1(\Sigma,p)$ appearing in
equation~(\ref{fundamental_group}), and we further identify
$\M=\mathcal{R}/G$ with the moduli space $\A(\Sigma)/\G(\Sigma)$:
\begin{equation}\label{identify_1}\begin{aligned}
\A(\Sigma)/\G(\Sigma,p) & \longleftrightarrow\mathcal{R}, \\
\A(\Sigma)/\G(\Sigma) & \longleftrightarrow\M=\mathcal{R}/G.
\end{aligned}\end{equation}
\noindent This identification is quite well known; see, for instance,
\cite{Jeffrey}.  Note that $\mathcal{R}$ may also be identified with the
set of homomorphisms from $\pi_1(\Sigma,p)$ to $G$.

Using a construction from \cite{Ho}, let $G\times G$ act on the
following subset $\widetilde{\mathcal{R}}$ of $G^{2(2-\chi(\Sigma))}$:
case(i),
\begin{equation}\begin{aligned}
& \widetilde{\mathcal{R}}=\left\{\begin{array}{l}
 (c,b,\AAA,\ol{c},\ol{b},\ol{\AAA})\in G^{2(2+k)} \ | \\
 \qquad b^{-1}cbc^{-1}=a_1\ol{a}_1\cdots a_k\ol{a}_k,\\
 \qquad \ol{b}^{-1}\ol{c}\ol{b}\ol{c}^{-1}=\ol{a}_1 a_1\cdots\ol{a}_k a_k
 \end{array}\right\}, \\
& (g,h).(c,b,\AAA,\ol{c},\ol{b},\ol{\AAA})
 =(gcg^{-1},gbg^{-1},g\AAA h^{-1},
 h\ol{c}h^{-1},h\ol{b}h^{-1},h\ol{\AAA}g^{-1});
\end{aligned}\end{equation}
\noindent case(ii),
\begin{equation}\begin{aligned}
& \widetilde{\mathcal{R}}=\left\{\begin{array}{l}
 (c,b,\AAA,\ol{c},\ol{b},\ol{\AAA})\in G^{2(2+k)} \ | \\
 \qquad \ol{b}^{-1}\ol{c}^{-1}\ol{b}c^{-1}
  =a_1\ol{a}_1\cdots a_k\ol{a}_k,\\
 \qquad b^{-1}c^{-1}b\ol{c}^{-1}
  =\ol{a}_1 a_1\cdots\ol{a}_k a_k \end{array}\right\}, \\
& (g,h).(c,b,\AAA,\ol{c},\ol{b},\ol{\AAA})
 =(gcg^{-1},gbh^{-1},g\AAA h^{-1},
 h\ol{c}h^{-1},h\ol{b}g^{-1},h\ol{\AAA}g^{-1});
\end{aligned}\end{equation}
\noindent case(iii),
\begin{equation}\begin{aligned}
& \widetilde{\mathcal{R}}=\left\{\begin{array}{l}
 (c,b,\AAA,\ol{c},\ol{b},\ol{\AAA})\in G^{2(2+2k)} \ | \\
 \qquad \ol{b}^{-1}\ol{c}^{-1}\ol{b}c^{-1}
  =[a_1,a_2]\cdots[a_{2k-1},a_{2k}],\\
 \qquad b^{-1}c^{-1}b\ol{c}^{-1}
  =[\ol{a}_1,\ol{a}_2]\cdots[\ol{a}_{2k-1},\ol{a}_{2k}]
 \end{array}\right\}, \\
& (g,h).(c,b,\AAA,\ol{c},\ol{b},\ol{\AAA})
 =(gcg^{-1},gbh^{-1},g\AAA g^{-1},
 h\ol{c}h^{-1},h\ol{b}g^{-1},h\ol{\AAA}h^{-1}).
\end{aligned}\end{equation}
\noindent Denote by $\G(\widetilde{\Sigma},\{P,\ol{P}\})$ the set of
gauge transformations on $\widetilde{\Sigma}$ that evaluate to 1 at
$P$ and at $\ol{P}$.  Each of the generators of $\pi_1(\Sigma,p)$
appearing in equation~(\ref{fundamental_group}) lifts to two curves in
$\widetilde{\Sigma}$.  By taking holonomies along the lifts of the
generators, we identify $\widetilde{\mathcal{R}}$ with
$\A(\widetilde{\Sigma})/\G(\widetilde{\Sigma},\{P,\ol{P}\})$, and we
further identify $\widetilde{\M}=\widetilde{\mathcal{R}}/(G\times G)$
with the moduli space $\A(\widetilde{\Sigma})/\G(\widetilde{\Sigma})$:
\begin{equation}\label{identify_2}\begin{aligned}
\A(\widetilde{\Sigma})/\G(\widetilde{\Sigma},\{P,\ol{P}\})
& \longleftrightarrow\widetilde{\mathcal{R}}, \\
\A(\widetilde{\Sigma})/\G(\widetilde{\Sigma})
& \longleftrightarrow\widetilde{\M}=\widetilde{\mathcal{R}}/(G\times G).
\end{aligned}\end{equation}
\noindent Compare the definition of $\widetilde{R}$ with
equation~(\ref{polygons}) in Remark~\ref{polygon_remark}.  To
understand the action of $G\times G$ on $\widetilde{\mathcal{R}}$,
consider how the holonomies along the lifts of the generators are
effected by a gauge transformation on $\widetilde{\Sigma}$ that
evaluates to $g$ at $P$ and to $h$ at $\ol{P}$.

\subsection{Some induced maps}\label{induced_maps}

Let $\rho$ and $\widetilde{\rho}$ be the natural quotient maps,
\begin{align*}
& \rho:\mathcal{R}\longrightarrow\M=\mathcal{R}/G, \\
& \widetilde{\rho}:\widetilde{\mathcal{R}}\longrightarrow
 \widetilde{\M}=\widetilde{\mathcal{R}}/(G\times G).
\end{align*}

The pullback of the covering map from $\widetilde{\Sigma}$ to
$\Sigma$ takes a Lie algebra valued 1-form on $\Sigma$ and
produces one on $\widetilde{\Sigma}$, and this pullback induces a
(well defined) map from $\A(\Sigma)/\G(\Sigma,p)$ to
$\A(\widetilde{\Sigma})/\G(\widetilde{\Sigma},\{P,\ol{P}\})$.  Using the
identifications appearing in equations~(\ref{identify_1})
and~(\ref{identify_2}), the pullback of the covering map induces a map
$\COVER:\mathcal{R}\rightarrow\widetilde{\mathcal{R}}$ given by
\begin{equation}
\COVER(c,b,\AAA)=(c,b,\AAA,c,b,\AAA).
\end{equation}
\noindent The map $\COVER$ satisfies $\COVER(g.(c,b,\AAA))=
(g,g).(\COVER(c,b,\AAA))$, and thus descends to a well-defined map
$\cover:\mathcal{M}\rightarrow\widetilde{\mathcal{M}}$.

Similarly, the pullback of the nontrivial deck transformation of
$\widetilde{\Sigma}$ induces an involution
$\DECK:\widetilde{\mathcal{R}}\rightarrow\widetilde{\mathcal{R}}$
given by
\begin{equation}
\DECK(c,b,\AAA,\ol{c},\ol{b},\ol{\AAA})
=(\ol{c},\ol{b},\ol{\AAA},c,b,\AAA).
\end{equation}
\noindent The map $\DECK$ satisfies
$\DECK((g,h).(c,b,\AAA,\ol{c},\ol{b},\ol{\AAA}))=
(h,g).(\DECK(c,b,\AAA,\ol{c},\ol{b},\ol{\AAA}))$, and thus descends to
a well-defined involution
$\deck:\widetilde{\mathcal{M}}\rightarrow\widetilde{\mathcal{M}}$.

The following commutative diagrams attempt to summarize these
definitions:
\begin{equation}
\begin{matrix}\text{Deck}\\ \text{transformation}\end{matrix}
\qquad\text{\LARGE$\rightsquigarrow$}\qquad
\begin{CD}
\widetilde{\mathcal{R}} @>\DECK>> \widetilde{\mathcal{R}} \\
@V{\widetilde{\rho}}VV @VV{\widetilde{\rho}}V\\
\widetilde{\mathcal{M}} @>\deck>> \widetilde{\mathcal{M}}
\end{CD}
\end{equation}
\begin{equation}\label{cover}
\begin{matrix}\text{Covering}\\ \text{map}\end{matrix}
\qquad\text{\LARGE$\rightsquigarrow$}\qquad
\begin{CD}
\mathcal{R} @>\COVER>> \widetilde{\mathcal{R}} \\
@V{\rho}VV @VV{\widetilde{\rho}}V \\
\mathcal{M} @>\cover>> \widetilde{\mathcal{M}}
\end{CD}
\end{equation}

\begin{remark}
Recall that $G=SU(2)$.  Although $\COVER$ is injective, the map
$\cover$ is not injective.  If $\rho(c,b,\AAA)\in\M$ and
\begin{equation}
\rho(f,e,\mathbb{D})=\begin{cases}
 \rho(c,b,-\AAA) & \text{case (i),} \\
 \rho(c,-b,-\AAA) & \text{case (ii),} \\
 \rho(c,-b,\AAA) & \text{case (iii)} \end{cases}
\end{equation}
\noindent then  $\cover(\rho(c,b,\AAA))=\cover(\rho(f,e,\mathbb{D}))$.
It can be shown that these two points of $\M$ are always distinct if
$\chi(\Sigma)$ is odd, and ``usually'' distinct if $\chi(\Sigma)$ is even:
in case~(i), for instance, if $\tr(a_1)\neq 0$ then
\begin{equation}
\rho(c,b,\AAA)\neq\rho(c,b,-\AAA)
\end{equation}
\noindent because $a_1$ is not conjugate to $-a_1$.
\end{remark}

\subsection{The fixed point set of $\deck$}\label{fixed_point}

Let $\widetilde{\mathcal{M}}^\deck$ be the fixed point set of $\deck$.
In \cite{Ho}, this fixed point set is shown to be a Lagrangian
submanifold of the moduli space $\widetilde{\M}$ with respect to the
symplectic form appearing in equation~(\ref{symplectic_form}).  The
following proposition, adapted from Section~3.1 of \cite{Ho}, allows us
to write $\widetilde{\mathcal{M}}^\deck$ as a union of more
manageable sets.

\begin{proposition}[N.-K. Ho]\label{proposition}
For each $x\in G$, define a subset
$\mathcal{N}_x\subset\widetilde{\mathcal{R}}$ as follows:
\begin{equation}\begin{aligned}
\text{case (i),}\quad & \mathcal{N}_x=
 \left\{(c,b,\AAA,c,b,\AAA x)\in\widetilde{\mathcal{R}} \ \ | \ \
 x.(c,b,\AAA)=(c,b,\AAA)\right\}; \\
\text{case (ii),}\quad & \mathcal{N}_x=
 \left\{(c,b,\AAA,c,bx,\AAA x)\in\widetilde{\mathcal{R}} \ \ | \ \
 x.(c,b,\AAA)=(c,b,\AAA)\right\}; \\
\text{case (iii),}\quad & \mathcal{N}_x=
 \left\{(c,b,\AAA,c,bx,\AAA)\in\widetilde{\mathcal{R}} \ \ | \ \
 x.(c,b,\AAA)=(c,b,\AAA)\right\}.
\end{aligned}\end{equation}
\noindent The image of $\cover$ is $\cover(\mathcal{M})=
\widetilde{\rho}(\mathcal{N}_1)$, and the fixed point set of $\deck$ is
$\widetilde{\M}^\deck=\bigcup\limits_{x\in G}
\widetilde{\rho}(\mathcal{N}_x)$.
\end{proposition}

\begin{proof}
To prove the first statement, note that
\begin{equation}
\mathcal{N}_1=\left\{(c,b,\AAA,c,b,\AAA) \ : \
(c,b,\AAA)\in\mathcal{R}\right\}=\COVER(\mathcal{R}),
\end{equation}
\noindent and use diagram~(\ref{cover}) to conclude that
$\widetilde{\rho}(\mathcal{N}_1)=\cover(\M)$.

We'll only prove the second statement for case~(i), since the proofs
are nearly identical for each of the three cases.  First note that
\begin{align}
& \widetilde{\rho}(c,b,\AAA,\ol{c},\ol{b},\ol{\AAA})
 \in\widetilde{\mathcal{M}}^\deck \notag\\
& \qquad \Longleftrightarrow \ \ \exists \ g,h\in G \ \text{such that} \
 \DECK(c,b,\AAA,\ol{c},\ol{b},\ol{\AAA})
 =(g,h).(c,b,\AAA,\ol{c},\ol{b},\ol{\AAA}) \\
& \qquad \Longleftrightarrow \ \ (\ol{c},\ol{b},\ol{\AAA},c,b,\AAA)
 =(gcg^{-1},gbg^{-1},g\AAA h^{-1},
 h\ol{c}h^{-1},h\ol{b}h^{-1},h\ol{\AAA}g^{-1}) \notag\\
& \qquad \Longleftrightarrow \ \ \exists \ g,h\in G \ \text{such that} \
 \left\{\begin{array}{l}
  \ol{c}=gcg^{-1} \\
  \ol{b}=gbg^{-1} \\
  \ol{\AAA}=g\AAA h^{-1} \\
  (hg)^{-1}c(hg)=c \\
  (hg)^{-1}b(hg)=b \\
  (hg)^{-1}(\AAA g)(hg)=\AAA g. \end{array}\right.
 \label{fixed_point_equation}
\end{align}
\noindent ($\subseteq$)
If $\widetilde{\rho}(c,b,\AAA,\ol{c},\ol{b},\ol{\AAA})
\in\widetilde{\mathcal{M}}^\deck$ then
equation~(\ref{fixed_point_equation}) implies that
\begin{equation}\begin{aligned}
(1,g^{-1}).(c,b,\AAA,\ol{c},\ol{b},\ol{\AAA})
& =(c,b,\AAA g,c,b,\AAA h^{-1}) \\
& =(c,b,\AAA g,c,b,(\AAA g)(hg)^{-1})\in\mathcal{N}_{(hg)^{-1}}.
\end{aligned}\end{equation}
\noindent ($\supseteq$)
If $(c,b,\AAA,c,b,\AAA x)\in\mathcal{N}_x$ then
\begin{align}
\DECK(c,b,\AAA,c,b,\AAA x)
&= (c,b,\AAA x,c,b,\AAA) \notag\\
&= (c,b,\AAA x,x^{-1}cx,x^{-1}bx,x^{-1}\AAA x) \\
&= (1,x^{-1}).(c,b,\AAA,c,b,\AAA x). \notag\qedhere
\end{align}
\end{proof}

\begin{remark}
An element $(c,b,\AAA,\ol{c},\ol{b},\ol{\AAA})$ of
$\widetilde{\mathcal{R}}$ lies in $\mathcal{N}_x$ if and only if
\begin{equation*}
\DECK(c,b,\AAA,\ol{c},\ol{b},\ol{\AAA})
=(1,x^{-1}).(c,b,\AAA,\ol{c},\ol{b},\ol{\AAA}).
\end{equation*}
\end{remark}

\begin{remark}
Recall that $G=SU(2)$.  In \cite{Ho}, an element
$\widetilde{\rho}(c,b,\AAA,\ol{c},\ol{b},\ol{\AAA})$ of $\widetilde{\M}$
is called \emph{generic} when the stabilizer of
$(c,b,\AAA,\ol{c},\ol{b},\ol{\AAA})\in\widetilde{\mathcal{R}}$
is $\{(1,1),(-1,-1)\}$.  Since an element of $\mathcal{N}_x$ is stabilized
by $(x,x)$, the fixed point set $\widetilde{\mathcal{M}}^\deck$ is
generically $\widetilde{\rho}(\mathcal{N}_1)\cup
\widetilde{\rho}(\mathcal{N}_{-1})$, and a straightforward computation
shows that $\widetilde{\rho}(\mathcal{N}_1)$ and
$\widetilde{\rho}(\mathcal{N}_{-1})$ are generically disjoint.  We shall
content ourselves with the description of
$\widetilde{\mathcal{M}}^\deck$ appearing in
Proposition~\ref{proposition}; arguing as in \cite{Ho}, however, it is
possible to show that $\widetilde{\mathcal{M}}^\deck=
\widetilde{\rho}(\mathcal{N}_1)\cup\widetilde{\rho}(\mathcal{N}_{-1})$.
\end{remark}

{\begin{figure}[h] 
\psfrag{c}{${\scriptstyle \gamma}$}
\psfrag{b}{${\scriptstyle \beta}$}
\psfrag{a}{${\scriptstyle \alpha_j}$}
\psfrag{U}{$U$}
\psfrag{Xi}{${\scriptstyle \text{flow}} \ \Xi_t^+$}
\includegraphics{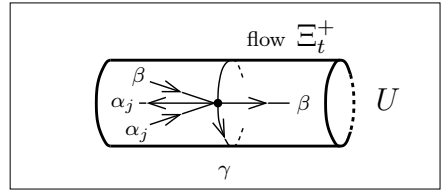}
\caption{\label{x8_tube}
The cylinder $U\subset\Sigma$.}
\end{figure}}

{\begin{figure}[h] 
\psfrag{C}{${\scriptstyle \Gamma}$}
\psfrag{0C}{${\scriptstyle \ol{\Gamma}}$}
\psfrag{B}{${\scriptstyle B}$}
\psfrag{0B}{${\scriptstyle \ol{B}}$}
\psfrag{A}{${\scriptstyle A_j}$}
\psfrag{0A}{${\scriptstyle \ol{A}_j}$}
\psfrag{V}{$V$}
\psfrag{0V}{$\ol{V}$}
\psfrag{Phi}{${\scriptstyle \text{flow}} \ \Phi_t^+$}
\psfrag{Psi}{${\scriptstyle \text{flow}} \ \Psi_t^-$}
\psfrag{case(i)}{$\text{case(i)}$}
\psfrag{case(ii)}{$\text{case(ii)}$}
\psfrag{case(iii)}{$\text{case(iii)}$}
\includegraphics{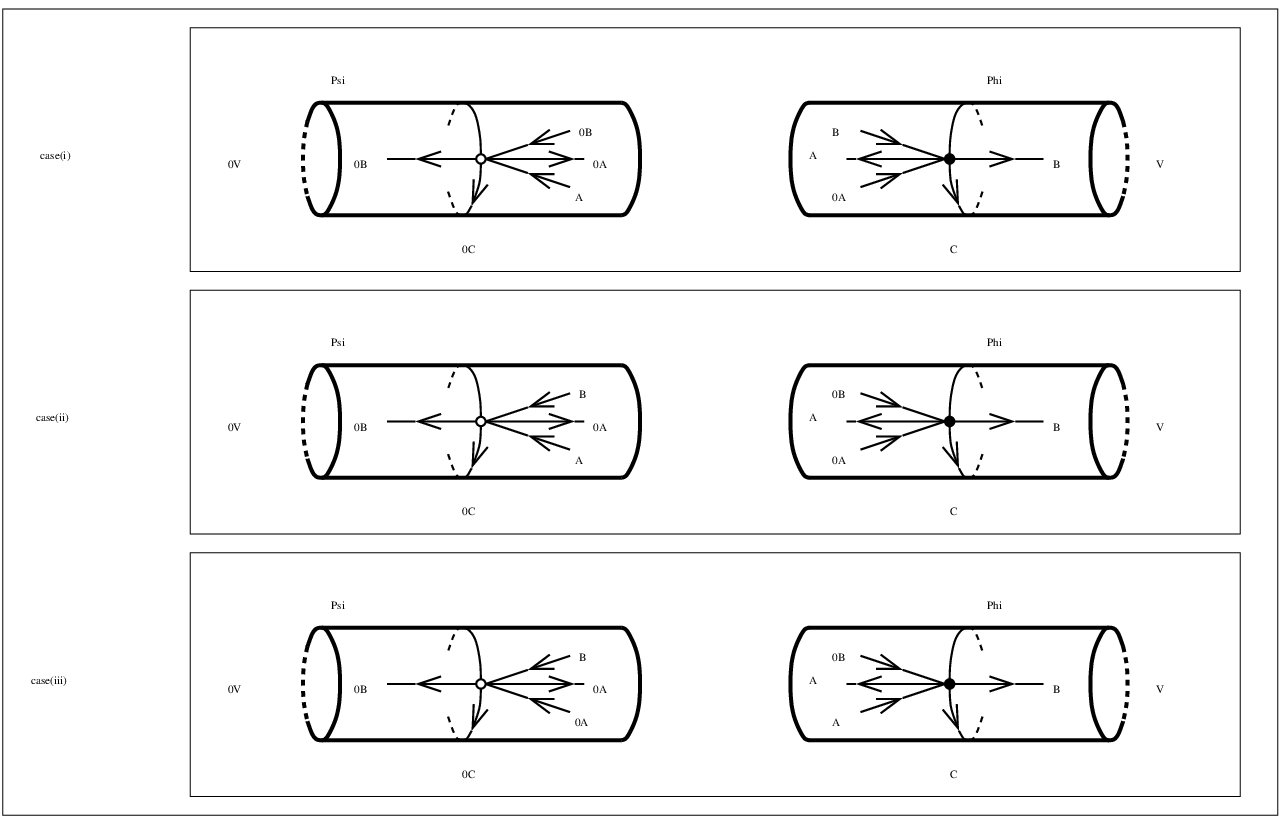}
\caption{\label{x9_tubes}
The cylinders $\ol{V},V\subset\widetilde{\Sigma}$.}
\end{figure}}

\subsection{The Goldman flows}\label{goldman_flow}

The cylindrical neighbourhood $U$ of $\gamma$ appears in
Figure~\ref{x8_tube}, and the cylindrical neighbourhoods $V$ of
$\Gamma$ and $\ol{V}$ of $\ol{\Gamma}$ are shown in
Figure~\ref{x9_tubes}.  Recall from Section~\ref{topology_cover} that
both $V$ and $\ol{V}$ are embedded in $\widetilde{\Sigma}$ in an
orientation preserving way, and that the covering map is orientation
preserving on $V$ and orientation reversing on $\ol{V}$.  In cases~(i)
and~(ii), each $A_j$ is a path from $P$ to $\ol{P}$, and each
$\ol{A}_j$ is a path from $\ol{P}$ to $P$.  In case (iii), each $A_j$ is a
loop based at $P$, and each $\ol{A}_j$ is a loop based at $\ol{P}$.

Under the identifications appearing in lines~(\ref{identify_1})
and~(\ref{identify_2}), if $(c,b,\AAA)\in\mathcal{R}$ then $c$
corresponds with the holonomy along $\gamma$, and if
$(c,b,\AAA,\ol{c},\ol{b},\ol{\AAA})\in\widetilde{\mathcal{R}}$ then $c$
and $\ol{c}$ correspond respectively with the holonomies along
$\Gamma$ and $\ol{\Gamma}$.  For $t\in\R$ and
$g\in G\smallsetminus\{\pm 1\}$, let
\begin{equation}
\zeta_t(g)=\exp(tF(g)),
\end{equation}
\noindent where $F$ is the the variation of $f$ defined in
equation~(\ref{variation}).  By comparing Figure~\ref{x8_tube} and
Figure~\ref{x9_tubes} with the table in Figure~\ref{x3_table}, we can
now describe the Goldman flows associated to each of the three
cylinders $U$, $V$ and $\ol{V}$.

\begin{definition}
Define an $\R$-action $\{\Xi_t^+\}_{t\in\R}$ on the following subset
$\mathcal{R}_\gamma$ of $\mathcal{R}$:
\begin{equation}\begin{aligned}
& \mathcal{R}_\gamma=\{(c,b,\AAA)\in\mathcal{R} \ | \ c\neq\pm 1\}, \\
& \Xi_t^+(c,b,\AAA)=(c,(\zeta_t(c))^{-1}b,\AAA).
\end{aligned}\end{equation}
\noindent The $\R$-action $\Xi_t^+$ covers the $2\pi$-periodic
Goldman flow $\Xi_t$ on $\M_\gamma=\mathcal{R}_\gamma/G$
associated to the cylinder $U$.
\end{definition}

\begin{definition}
Define an $\R$-action $\{\Phi_t^+\}_{t\in\R}$ on the following subset
$\widetilde{\mathcal{R}}_\Gamma$ of $\widetilde{\mathcal{R}}$:
\begin{equation}\begin{aligned}
& \widetilde{\mathcal{R}}_{\Gamma}=
 \{(c,b,\AAA,\ol{c},\ol{b},\ol{\AAA})\in\widetilde{\mathcal{R}} \ | \
 c\neq\pm 1\}, \\
& \Phi_t^+(c,b,\AAA,\ol{c},\ol{b},\ol{\AAA})=
 (c,(\zeta_t(c))^{-1}b,\AAA,\ol{c},\ol{b},\ol{\AAA}).
\end{aligned}\end{equation}
\noindent The $\R$-action $\Phi_t^+$ covers the $2\pi$-periodic
Goldman flow $\Phi_t$ on $\widetilde{\M}_\Gamma=
\widetilde{\mathcal{R}}_\Gamma/(G\times G)$ associated to the
cylinder $V$.
\end{definition}

\begin{definition}
Define an $\R$-action $\{\Psi_t^-\}_{t\in\R}$ on the following subset
$\widetilde{\mathcal{R}}_{\ol{\Gamma}}$ of $\widetilde{\mathcal{R}}$:
\begin{equation}\begin{aligned}
& \widetilde{\mathcal{R}}_{\ol{\Gamma}}=
 \{(c,b,\AAA,\ol{c},\ol{b},\ol{\AAA})\in\widetilde{\mathcal{R}} \ | \
 \ol{c}\neq\pm 1\}, \\
& \Psi_{t}^-(c,b,\AAA,\ol{c},\ol{b},\ol{\AAA})=
 (c,b,\AAA,\ol{c},(\zeta_t(\ol{c}))\ol{b},\ol{\AAA}).
\end{aligned}\end{equation}
\noindent The $\R$-action $\Psi_t^-$ covers the $2\pi$-periodic
Goldman flow $\Psi_t$ on $\widetilde{\M}_{\ol{\Gamma}}=
\widetilde{\mathcal{R}}_{\ol{\Gamma}}/(G\times G)$ associated to the
cylinder $\ol{V}$.
\end{definition}

\begin{remark}
The actions $\Phi_t^+$ and $\Psi_t^-$ commute because the cylinders
$V$ and $\ol{V}$ are disjoint.  The composition
$\Phi_t^+\circ\Psi_{-t}^-$ defines an $\R$-action on
$\widetilde{\mathcal{R}}_\Gamma\cap
\widetilde{\mathcal{R}}_{\ol{\Gamma}}$,
\begin{equation}
(\Phi_t^+\circ\Psi_{-t}^-)(c,b,\AAA,\ol{c},\ol{b},\ol{\AAA})=
(c,(\zeta_t(c))^{-1}b,\AAA,\ol{c},(\zeta_t(\ol{c}))^{-1}\ol{b},\ol{\AAA}),
\end{equation}
\noindent which covers a $2\pi$-periodic $\R$-action
$\Phi_t\circ\Psi_{-t}$ on
$\widetilde{\M}_\Gamma\cap\widetilde{\M}_{\ol{\Gamma}}$.  In the
introduction, this flow was referred to as the \emph{composite flow}.
\end{remark}

\subsection{Proof of the main theorem}\label{proof}

In this section we prove that the composite flow preserves not only the
Langrangian submanifold $\widetilde{\M}^\deck$, but also the image of
the map $\cover$.

\begin{proposition}\label{prop1}
$(\Phi_t\circ\Psi_{-t})\circ\deck=\deck\circ(\Phi_t\circ\Psi_{-t})$.
\end{proposition}

\begin{proof} It suffices to show that
$(\Phi_t^+\circ\Psi_{-t}^-)\circ\DECK=
\DECK\circ(\Phi_t^+\circ\Psi_{-t}^-)$.
\begin{align}
((\Phi_t^+\circ\Psi_{-t}^-)\circ\DECK)(c,b,\AAA,\ol{c},\ol{b},\ol{\AAA})
& =(\Phi_t^+\circ\Psi_{-t}^-)(\ol{c},\ol{b},\ol{\AAA},c,b,\AAA) \\
& =(\ol{c},(\zeta_t(\ol{c}))^{-1}\ol{b},\ol{\AAA},
 c,(\zeta_t(c))^{-1}b,\AAA) \notag\\
& =\DECK(c,(\zeta_t(c))^{-1}b,\AAA,
 \ol{c},(\zeta_t(\ol{c}))^{-1}\ol{b},\ol{\AAA}) \notag\\
& =(\DECK\circ(\Phi_t^+\circ\Psi_{-t}^-))
 (c,b,\AAA,\ol{c},\ol{b},\ol{\AAA}). \notag\qedhere
\end{align}
\end{proof}

\begin{proposition}\label{prop2}
$(\Phi_t\circ\Psi_{-t})\circ\cover=\cover\circ\Xi_t$.
\end{proposition}

\begin{proof} It suffices to show that
$(\Phi_t^+\circ\Psi_{-t}^-)\circ\COVER=\COVER\circ\Xi_t^+$.
\begin{align}
((\Phi_t^+\circ\Psi_{-t}^-)\circ\COVER)(c,b,\AAA)
& =(\Phi_t^+\circ\Psi_{-t}^-)(c,b,\AAA,c,b,\AAA) \\
& =(c,(\zeta_t(c))^{-1}b,\AAA,c,(\zeta_t(c))^{-1}b,\AAA) \notag\\
& =\COVER(c,(\zeta_t(c))^{-1}b,\AAA) \notag\\
& =(\COVER\circ\Xi_t^+)(c,b,\AAA). \notag\qedhere
\end{align}
\end{proof}

\begin{lemma}\label{lemma}
Recall that $\cover(\M)=\widetilde{\rho}(\mathcal{N}_1)$ and
$\widetilde{\M}^\deck=\bigcup\limits_{x\in G}
\widetilde{\rho}(\mathcal{N}_x)$.  The flow $\Phi_t\circ\Psi_{-t}$ on
$\widetilde{\M}_\Gamma\cap\widetilde{\M}_{\ol{\Gamma}}$ preserves
$\widetilde{\rho}(\mathcal{N}_x)$.
\end{lemma}

\begin{proof}
It suffices to show that the flow $\Phi_t^+\circ\Psi_{-t}^-$ on
$\widetilde{\mathcal{R}}_\Gamma\cap
\widetilde{\mathcal{R}}_{\ol{\Gamma}}$ preserves $\mathcal{N}_x$.
We'll only give the proof for case~(i), since the proofs are nearly
identical for each of the three cases.  Suppose
$(c,b,\AAA,\ol{c},\ol{b},\ol{\AAA})=(c,b,\AAA,c,b,\AAA x)
\in\mathcal{N}_x\cap(\widetilde{\mathcal{R}}_\Gamma
\cap\widetilde{\mathcal{R}}_{\ol{\Gamma}})$, and note that
\begin{equation}
(\Phi_t^+\circ\Psi_{-t}^-)(c,b,\AAA,c,b,\AAA x)
=(c,(\zeta_t(c))^{-1}b,\AAA,c,(\zeta_t(c))^{-1}b,\AAA x).
\end{equation}
\noindent Since $x$ commutes with both $c$ and $b$,
equation~(\ref{variation_commute}) implies that $\Ad{x}F(c)=F(c)$,
\begin{equation}\begin{aligned}
& \Rightarrow \ x(\zeta_t(c))x^{-1}=\zeta_t(c) \\
& \Rightarrow \ x(\zeta_t(c))^{-1}bx^{-1}=(\zeta_t(c))^{-1}b.
\end{aligned}\end{equation}
\noindent Thus $(c,(\zeta_t(c))^{-1}b,\AAA,c,(\zeta_t(c))^{-1}b,\AAA x)
\in\mathcal{N}_x$, and the proof is complete.
\end{proof}

\begin{theorem}[Main Theorem]\label{main_theorem}
Consider the Goldman flow $\Xi_t$ on $\mathcal{M}_\gamma$ and the
composite flow $\Phi_t\circ\Psi_{-t}$ on
$\widetilde{\mathcal{M}}_\Gamma\cap
\widetilde{\mathcal{M}}_{\ol{\Gamma}}$.
\begin{itemize}
\item[(i)] $\Phi_t\circ\Psi_{-t}$ preserves $\widetilde{\M}^\deck$,
\item[(ii)] $\Phi_t\circ\Psi_{-t}$ preserves $\cover(\M)$,
\item[(iii)] $(\Phi_t\circ\Psi_{-t})\circ\deck
 =\deck\circ(\Phi_t\circ\Psi_{-t})$,
\item[(iv)] $(\Phi_t\circ\Psi_{-t})\circ\cover=\cover\circ\Xi_t$.
\end{itemize}
\end{theorem}

\begin{proof}
We have already proved parts~(iii) and~(iv).  Part~(i) follows either
from part~(iii) or from Lemma~\ref{lemma}.  Part~(ii) follows either
from part~(iv) or from Lemma~\ref{lemma}.
\end{proof}

\begin{remark}
If the nonorientable surface $\Sigma$ is \emph{noncompact} then
Theorem~\ref{main_theorem} is still true, and the gauge theoretic
description of the Goldman flow given in Section~\ref{section1} can be
used to produce a proof.  Let $\eta(s)$ be a bump function with
support in $(0,1)$.  As in the proof of Theorem~\ref{theorem_flow}, use
$\eta(s)$ to define both the flow $\Xi_t^+$ associated to
$U\subset\Sigma$ and the flow $\Phi_t^+$ associated to
$V\subset\widetilde{\Sigma}$, and use $\eta(-s)$ to define the flow
$\Psi_t^-$ associated to $\ol{V}\subset\widetilde{\Sigma}$.  Recall
from the proof of Theorem~\ref{theorem_holonomy} that if $A$ is a flat
connection on $\Sigma$ then there are gauge transformations $\xi_t$
on $\Sigma\smallsetminus\gamma$ such that
\begin{equation}
\xi_t.(A|_{\Sigma\smallsetminus\gamma})
=(\Xi_t^+ A)|_{\Sigma\smallsetminus\gamma}.
\end{equation}
\noindent Similarly, if $A$ is a flat connection on $\widetilde{\Sigma}$
then there are gauge transformations
$\varphi_t\in\G(\widetilde{\Sigma}\smallsetminus\Gamma)$ and
$\psi_t\in\G(\widetilde{\Sigma}\smallsetminus\ol{\Gamma})$ such that
\begin{equation}\begin{aligned}
& \varphi_t.(A|_{\Sigma\smallsetminus\Gamma})
 =(\Phi_t^+ A)|_{\Sigma\smallsetminus\Gamma}, \\
& \psi_t.(A|_{\Sigma\smallsetminus\ol{\Gamma}})
 =(\Psi_t^- A)|_{\Sigma\smallsetminus\ol{\Gamma}}.
\end{aligned}\end{equation}
\noindent The gauge transformations $\xi_t$, $\varphi_t$, and $\psi_t$
can be used to prove Proposition~\ref{prop1} and
Proposition~\ref{prop2}.
\end{remark}


\begin{thebibliography}{MM}

\bibitem{AtiyahBott} M. F. Atiyah and R. Bott,
\emph{The Yang-Mills equation over Riemann surfaces},
Philos.\ Trans.\ Roy.\ Soc.\ London Ser.\ A
{\bf 308} (1983), no.\ 1505, 523--615.

\bibitem{Donaldson} S. K. Donaldson,
\emph{Gluing techniques in the cohomology of moduli spaces},
Topological Methods in Modern Mathematics, 137--170,
Publish or Perish, Houston, TX, 1993.

\bibitem{Goldman86} W. Goldman,
\emph{Invariant functions on Lie groups and Hamiltonian flows of
suface group representations},
Invent.\ Math.\ {\bf 85} (1986), no.\ 2, 263--302.

\bibitem{Goldman97} W. Goldman,
\emph{Ergodic theory on moduli spaces},
Ann.\ of Math.\ (2) {\bf 146} (1997), no.\ 3, 475--507.

\bibitem{Goldman04} W. Goldman,
\emph{The complex-symplectic geometry of
$SL(2,\mathbb{C})$-characters over surfaces},
Algebraic groups and arithmetic, 375--407,
Tata Inst.\ Fund.\ Res., Mumbai, 2004.

\bibitem{Ho} N.-K. Ho,
\emph{The real locus of an involution map on the moduli space of flat
connections on a Riemann surface},
Int.\ Math.\ Res.\ Not.\ {\bf 2004}, no.\ 61, 3263--3285.

\bibitem{Jeffrey} L. C. Jeffrey,
\emph{Flat connections on oriented 2-manifolds},
Bull.\ London Math.\ Soc.\ {\bf 37} (2005), no.\ 1, 1--14.

\bibitem{JeffreyWeitsman} L. C. Jeffrey and J. Weitsman,
\emph{Bohr-Sommerfeld orbits in the moduli space of flat connections
and the Verlinde dimension formula},
Commun.\ Math.\ Phys.\ {\bf 150} (1992), no.\ 3, 593--630.

\bibitem{McDuffSalamon}D. McDuff and D. Salamon,
\emph{Introduction to symplectic topology}.
Second edition.  Oxford Mathematical Monographs.
The Clarendon Press, Oxford University Press, New York, 1998.

\bibitem{Tyurin}A. Tyurin,
\emph{Delzant models of moduli spaces},
Izv.\ Math.\ {\bf 67} (2003), no.\ 2, 365--376.

\end{thebibliography}
\end{document}